\documentclass[12pt,a4paper]{amsart}
\usepackage{a4wide}
\usepackage{amsmath, amssymb, amsfonts,enumerate}
\usepackage[all]{xy}
\usepackage{amscd}
\usepackage{comment}
\usepackage{mathtools}

\newtheorem{theorem}{Theorem}[section]
\newtheorem{lemma}[theorem]{Lemma}
\newtheorem{corollary}[theorem]{Corollary}
\newtheorem{proposition}[theorem]{Proposition}

 \theoremstyle{definition}
 
 \newtheorem{remark}[theorem]{Remark}

 \newtheorem{example}[theorem]{Example}

\numberwithin{equation}{section}
\newcommand {\N}{\mathbb{N}} 
\newcommand {\Z}{\mathbb{Z}} 
\newcommand {\R}{\mathbb{R}} 
\newcommand {\Q}{\mathbb{Q}} 
\newcommand {\C}{\mathbb{C}} 

\newcommand{\T}{\mathbb{T}}

\newcommand{\FF}{\mathcal{F}}

   \DeclareMathOperator{\Unit}{U}

\DeclareMathOperator{\Id}{Id}

\DeclareMathOperator{\ann}{Ann}

\begin{document}
\title[Surjunctivity and topological rigidity]{Surjunctivity and topological rigidity of algebraic dynamical systems}
\author[S. Bhattacharya]{Siddhartha Bhattacharya}
\address{School of Mathematics,
Tata Institute of Fundamental Research, Mumbai 400005, India}
\email{siddhart@math.tifr.res.in}
\author[T.Ceccherini-Silberstein]{Tullio Ceccherini-Silberstein}
\address{Dipartimento di Ingegneria, Universit\`a del Sannio, C.so
Garibaldi 107, 82100 Benevento, Italy}
\email{tullio.cs@sbai.uniroma1.it}
\author[M. Coornaert]{Michel Coornaert}
\address{Universit\'e de Strasbourg, CNRS, IRMA UMR 7501, F-67000 Strasbourg, France}
\email{michel.coornaert@math.unistra.fr}
\subjclass[2010]{37D20,  22D45, 22E40, 54H20, 37A25}
\keywords{algebraic dynamical system, compact group, surjunctivity, expansivity, mixing, descending chain condition, topological rigidity,  entropy}
\begin{abstract}
Let $X$ be a compact  metrizable  group and  $\Gamma$  a countable group acting  on $X$ by continuous group automorphisms.
We give sufficient conditions under which the dynamical system $(X,\Gamma)$ is  surjunctive, i.e.,
every injective continuous map $\tau \colon X \to X$ commuting with the action of 
$\Gamma$  is surjective.
\end{abstract}
\date{\today}
\maketitle

\section{Introduction}

Consider a dynamical system $(X,\Gamma)$ consisting of a compact metrizable space $X$ equipped with a continuous action $\alpha$ of a countable group $\Gamma$.
This means that $\alpha$ is a group morphism from $\Gamma$ into the group of homeomorphisms of $X$.
To simplify notation, let us write $\alpha(\gamma)(x) = \gamma x$ for all $\gamma \in \Gamma$ and $x \in X$.
A map $\tau \colon X \to X$ is said to be \emph{$\Gamma$-equivariant} if it commutes with the action of $\Gamma$ on $X$, i.e., 
$\tau(\gamma x) = \gamma \tau(x)$ for all $\gamma \in \Gamma$ and $x \in X$.
Following a terminology introduced by Gottschalk in \cite{gottschalk},
we say that the dynamical system $(X,\Gamma)$ is \emph{surjunctive} if every 
injective $\Gamma$-equivariant continuous map
$\tau \colon X \to X$ is surjective (and hence a homeomorphism).
\par
There are several important classes of dynamical systems that are known to be surjunctive.
\par
First note that the dynamical system  $(X,\Gamma)$ is surjunctive whenever  the phase space $X$ is \emph{incompressible}, i.e., there is no proper subset of $X$ 
that is  homeomorphic to $X$.
This is for example the case when $X$ is a closed topological manifold (e.g. a compact Lie group) by Brouwer's invariance of domain.
\par
Another class of surjunctive dynamical systems is provided  by the systems that satisfy the descending chain condition.
One says that the dynamical system  $(X,\Gamma)$ satisfies the \emph{descending chain condition} (\emph{d.c.c.}\ for short)  if every decreasing sequence
$$
X = X_0 \supset X_1 \supset X_2 \supset \dots
$$
of $\Gamma$-invariant closed subsets $X_i \subset X$ ($i \in \N$) eventually stabilizes, i.e., there is an integer $k \geq 0$ such that
$X_i = X_k$ for all $i \geq k$.
Indeed, suppose that $(X,\Gamma)$ satisfies the d.c.c.\ and $\tau \colon X \to X$ is an injective 
$\Gamma$-equivariant continuous map.
  By applying the d.c.c.\ to the sequence
  $$
  X= \tau^0(X) \supset \tau(X) = \tau^1(X) \supset \tau^2(X) \supset \dots,
  $$
  we see that there is an integer $k \geq 0$ such that
  $\tau^k(X) = \tau^{k + 1}(X)$.
Since $\tau^k$ is injective, this implies 
$X = \tau(X)$, showing that $\tau$ is surjective.
  Note that a minimal dynamical system
  (i.e., a system containing no proper invariant closed subsets)
   trivially satisfies the d.c.c.\
  Actually every $\Gamma$-equivariant continuous map $\tau \colon X \to X$ is surjective when 
  $(X,\Gamma)$ is minimal.
  More generally, $(X,\Gamma)$ satisfies the d.c.c.\ if $X$ contains only finitely many closed 
  $\Gamma$-invariant subsets or if every proper closed $\Gamma$-invariant subset of $X$ is finite. 
  \par
There are also several classes of symbolic dynamical systems that are known to be surjunctive. Let $S$ be a  compact metrizable 
space, called the \emph{alphabet} or the \emph{space of symbols}.
   Given a countable group $\Gamma$,  the 
  $\Gamma$-\emph{shift} on the alphabet $S$ is the dynamical system   $(S^\Gamma, \Gamma)$,
  where $S^\Gamma = \{x \colon \Gamma \to S\}$ is equipped with the product  topology  and the action of $\Gamma$ on 
  $S^\Gamma$ is given by the formula $(\gamma x)(\gamma') \coloneqq x(\gamma^{-1} \gamma')$ for 
  all $\gamma,\gamma' \in \Gamma$ and $x \in S^\Gamma$. 
These shift systems are not surjunctive in general.
Indeed, suppose  for example that the alphabet space $S$ is compressible (e.g. $S$ is the closed unit interval $[0,1] \subset \R$,
or the Cantor set $K \subset [0,1]$, or the infinite-dimensional torus
$\T^\N$ with $\T \coloneqq \Unit(1)$ the unit circle group) and let $\iota \colon S \to S$ be an injective continuous map that is not surjective. 
Then the map $\tau \colon S^\Gamma \to S^\Gamma$ given by
$\tau(x)(\gamma) \coloneqq \iota(x(\gamma))$ for all $x \in S^\Gamma$ and $\gamma \in \Gamma$,
is clearly   injective, $\Gamma$-equivariant, and continuous.
However, $\tau$ is  not surjective since $\iota$ is not.
Thus, the shift system $(S^\Gamma,\Gamma)$ is never surjunctive when the alphabet  $S$ is compressible.
If $M$ is a closed topological manifold and the group $\Gamma$ is residually finite, 
it was observed in \cite[Corollary~7.8]{csc-concrete} that  the 
$\Gamma$-shift
$(M^\Gamma,\Gamma)$ is surjunctive.
 On the other hand, it follows from a deep theorem of Gromov~\cite{gromov-esav} and 
 Weiss~\cite{weiss-sgds} that,   if $\Gamma$ is a sofic group and $S$ is a finite discrete space, then
 the $\Gamma$-shift  $(S^\Gamma,\Gamma)$  is surjunctive.
Sofic groups form a wide class of groups containing in particular all residually amenable  groups
 but it is not known if the Gromov-Weiss surjunctivity theorem remains valid  for all groups (\emph{Gottschalk conjecture}). Actually, there is no example of a non-sofic group up to now
 although  many experts in the field do believe in the existence of non-sofic groups.
 \par
 One says that a dynamical system  $(X,\Gamma)$ is \emph{expansive} if there exists a constant 
$\delta = \delta(X,\Gamma,d)  > 0$ such that, for every pair of distinct points $x,y \in X$, there exists an element
$\gamma = \gamma(x,y)  \in \Gamma$ such that
$d(\gamma  x,\gamma  y) \geq  \delta$. Here $d$ denotes a compatible metric on $X$.
The fact that this  definition does not depend on the choice of $d$
follows from the  compactness of $X$.
For instance, a shift system $(S^\Gamma,\Gamma)$ is expansive if and only if the alphabet $S$ is finite.
\par
Given a dynamical system $(X,\Gamma)$, a point $x \in X$ is said to be $\Gamma$-\emph{periodic}
if its $\Gamma$-orbit $\Gamma x \coloneqq  \{\gamma x : \gamma \in \Gamma\} \subset X$ is finite.
In~\cite[Proposition~5.1]{csc-ijm-2015}, it was observed that if $(X,\Gamma)$ is an expansive dynamical system whose periodic points are dense in  $X$,  then $(X,\Gamma)$ is surjunctive.
\par
Note that satisfying the d.c.c.\  is a hereditary property, in the sense that every subsystem of a  dynamical system satisfying the d.c.c.\ 
satisfies it as well. Expansivity is also hereditary 
but surjunctivity and density of periodic points are not. 
Consider for example the $\Z$-shift $(\{0,1\}^\Z,\Z)$
on the alphabet with two elements $0$ and $1$, 
and the  closed $\Z$-invariant subset  $\Sigma \subset \{0,1\}^\Z$ consisting of 
all bi-infinite sequences of $0$s and $1$s  with at most one chain of $1$s.
Then the continuous map $\tau \colon \Sigma \to \Sigma$ given by
$$
\tau(x)(n) =
\begin{cases}
1 & \text{if } (x(n),x(n+1)) = (0,1) \\
x(n) & \text{otherwise},
\end{cases}   
$$
is clearly  injective and commutes with the shift.
However, $\tau$ is not surjective since
a sequence of $0$s and $1$s where $1$ appears exactly once is in $\Sigma$ but not in the image 
of $\tau$.
Thus, the subsystem $(\Sigma,\Z)$ is not surjunctive.
On the other hand, 
periodic sequences are dense in $\{0,1\}^\Z$
while the only periodic sequences in $\Sigma$ are the two constant ones.
This example also shows that an expansive dynamical system may fail to be surjunctive.
\par
An \emph{algebraic dynamical system}
is a dynamical system of the form $(X, \Gamma)$, where $X$ is a compact metrizable topological group and $\Gamma$ is a countable group acting on $X$  by continuous group automorphisms.
Observe that if $S$ is a compact metrizable group, then the $\Gamma$-shift $(S^\Gamma,\Gamma)$ is an algebraic dynamical system for every countable group $\Gamma$. 
\par
Algebraic dynamical systems have been intensively studied in the last decades
(see  the monograph~\cite{schmidt-book}
the survey~\cite{lind-schmidt-survey-heisenberg},  and the references therein). As it was already observed by Halmos~\cite{halmos}  in the particular case $\Gamma = \Z$,
they provide a large supply of interesting examples for ergodic theory.
When the phase space $X$ of an algebraic dynamical system $(X,\Gamma)$ is abelian, the action of 
$\Gamma$ on $X$ induces a 
$\Z[\Gamma]$-module structure on $X$, where $\Z[\Gamma]$ denotes the integral group ring 
of $\Gamma$.
The Pontryagin dual $\widehat{X}$ of $X$, which is  a countable discrete abelian group,
is then endowed with a dual  $\Z[\Gamma]$-module structure.
This yields a one-to-one correspondence between the class of algebraic dynamical systems with abelian phase space and the class of countable  $\Z[\Gamma]$-modules.
This correspondence has been fruitfully used by translating topological properties of these  system into
algebraic properties of the associated dual $\Z[\Gamma]$-module
especially in the case $\Gamma = \Z^d$ 
(see e.g.~\cite{kitchens-schmidt-1989}, \cite{schmidt-book}, \cite{chung-li},
\cite{lind-schmidt-survey-heisenberg})
and revealed in particular amazing connections between the theory of algebraic dynamical systems and number theory.  
\par
The goal of the present  paper is to give a series of sufficient conditions for an algebraic dynamical system to be surjunctive. 
 The first of our  results is the following.

\begin{theorem}
\label{t:th1}
Let $X$ be a  compact metrizable  group equipped with an action of a countable group $\Gamma$ by continuous group automorphisms. 
Suppose that $X$ is connected with finite topological dimension
and that the action of $\Gamma$ on $X$ is expansive. 
Then the dynamical system $(X,\Gamma)$ is surjunctive.
\end{theorem}

\begin{remark}
\label{r:lam}
By a result of Lam \cite[Theorem~3.2]{lam},
if a compact connected metrizable group $X$ admits an expansive action of a countable group 
$\Gamma$ by continuous group automorphisms,  then $X$ is necessarily abelian.
Thus, every group $X$ that satisfies  the hypotheses of 
 Theorem~\ref{t:th1} must be  abelian.
 Finite-dimensional connected metrizable abelian groups are often called \emph{solenoids} in the literature.
 They include all finite-dimensional tori $\T^  n$ but also groups that are not Lie such as the
 $2$-\emph{adic solenoid}, which 
is the projective limit of an infinite sequence of circles, where each circle is wrapped twice around its predecessor.  
 By Pontryagin duality, solenoids are in one-to-one correspondance with finite-rank torsion-free abelian groups (see e.g.~\cite{morris} and
 Subsection~\ref{subsec:pontryagin} below).
\end{remark}

Our second result is the following.

\begin{theorem}
\label{t:th2}
Let $X$ be a (possibly non-abelian) compact metrizable group equipped with an action of  $\Z^d$ by continuous group automorphisms.
Suppose that the  action of $\Z^d$ on $X$ is expansive.
Then the dynamical system $(X,\Z^d)$ is surjunctive.
\end{theorem}

As every shift over a finite alphabet is expansive and any subsystem of an expansive system is itself expansive,
an immediate consequence of Theorem~\ref{t:th2} is the following.

\begin{corollary}
Let $S$ be a (possibly non-abelian) finite discrete group and let $X \subset S^{\Z^d}$ be a closed subgroup that is invariant under the $\Z^d$-shift.
Then the algebraic dynamical system $(X,\Z^d)$, where the action of $\Z^d$ on $X$ is the one induced by restriction of the $\Z^d$-shift,  is surjunctive.
\qed 
\end{corollary}

    One says that an algebraic  dynamical system $(X,\Gamma)$ satisfies the \emph{algebraic descending chain condition} (\emph{a.d.c.c.}\ for short)  
    if every decreasing sequence
$$
X = X_0 \supset X_1 \supset X_2 \supset \dots
$$
of $\Gamma$-invariant closed subgroups $X_i \subset X$ ($i \in \N$) eventually stabilizes.
In the case when $X$ is abelian, this condition is equivalent to the fact that its Pontryagin dual 
$\widehat{X}$ is Noetherian as a 
$\Z[\Gamma]$-module .
By a result of Kitchens and Schmidt \cite[Theorem~5.2]{kitchens-schmidt-1989},
every expansive algebraic dynamical system $(X,\Z^d)$ satisfies the algebraic descending chain condition.
On the other hand, if $X$ is any  infinite compact Lie group,   then the $\Z^d$-shift on $X^{\Z^d}$ satisfies the a.d.c.c.\ \cite[Theorem~3.2]{kitchens-schmidt-1989} without being expansive.
Our next result is the following.  

\begin{theorem}
\label{t:th3}
Let $X$ be a compact metrizable abelian group equipped with an action of  
 $\Z^d$ by continuous group automorphisms.
Suppose that $(X,\Z^d)$ satisfies the algebraic descending chain condition.
Then the dynamical system $(X,\Z^d)$ is surjunctive.
\end{theorem}

If  $(X,\Gamma)$ is an algebraic dynamical system, then the  Haar probability measure $\lambda_X$ on $X$ is $\Gamma$-invariant, i.e., $\lambda_X(\gamma A) = \lambda_X(A)$ for every measurable subset 
$A \subset X$ (this imediately follows from the uniqueness of Haar measure).
One says that $(X,\Gamma)$ is \emph{mixing} if
\begin{equation}
\label{e:def-mixing}
\lim\limits_{\gamma \to\infty}
\lambda_{X}(A_{1} \cap \gamma A_{2})
=\lambda_{X}(A_{1})\cdot\lambda_{X}(A_{2})
\end{equation}
for all measurable subsets $A_1,A_2 \subset X$ (here $\infty$ denotes the point at infinity in the 
one-point compactification of the discrete group $\Gamma$).
For example, the $\Gamma$-shift $(S^\Gamma,\Gamma)$ is mixing for any 
compact metrizable group $S$ and any countable group $\Gamma$  (just observe 
that~\eqref{e:def-mixing}  is obvious when 
$A_1, A_2 \subset S^\Gamma$ are cylinders).
\par
One says that a countable group $\Gamma$ satisfies the
\emph{$\ell^2$-zero-divisor conjecture} if the space $\ell^2(\Gamma)$ of square-summable complex-valued functions on $\Gamma$  is torsion-free as a $\C[\Gamma]$-module 
(see Section~\ref{sec:l2-zdc} below).
Linnell~\cite[Theorem~2]{linnell-zdc-pacific}  proved that every torsion-free elementary amenable group satisfies the $\ell^2$-zero-divisor conjecture.
We recall that the class of \emph{elementary amenable groups} is the smallest class of groups which contains all abelian  and all finite groups and is closed under extensions and    
directed unions. 
We shall prove the following.

\begin{theorem}
\label{t:th4}
Let $X$ be a compact connected  metrizable abelian group
equipped with an action of a countable group $\Gamma$ by continuous group automorphisms.
Suppose that the following conditions hold:
\begin{enumerate}[{\rm (i)}]
\item $\Gamma$ satisfies the $\ell^2$-zero-divisor conjecture
(e.g. $\Gamma$ is elementary amenable and torsion-free);
\item $(X,\Gamma)$ is mixing;
\item $\widehat{X}$ is a  Noetherian torsion $\Z[\Gamma]$-module.
\end{enumerate}
Then the dynamical system $(X,\Gamma)$ is surjunctive.
\end{theorem}

Finally, we shall establish the following.

\begin{theorem}
\label{t:poly-abel}
Let $X$ be a compact connected  metrizable abelian group
equipped with an action of a countable group $\Gamma$ by continuous group automorphisms.
Suppose that the following conditions are satisfied:
\begin{enumerate}[{\rm (i)}]
\item $\Gamma$ is torsion-free and either polycyclic-by-finite or abelian;
\item $(X,\Gamma)$ is mixing;
\item $\widehat{X}$ is a finitely generated torsion $\Z[\Gamma]$-module.
\end{enumerate}
Then the dynamical system $(X,\Gamma)$ is surjunctive.
\end{theorem}

The paper is organized as follows.
Section~\ref{sec:background} introduces notation and collects some background material and preliminaries.
In Section~\ref{sec:l2-zdc}, we establish a topological rigidity  result for algebraic actions of groups 
$\Gamma$ that satisfy the $\ell^2$-zero-divisor conjecture. 
This rigidity result, which is of independent interest,  is contained in 
\cite[Theorem~1.1]{bhattacharya-ward} for $\Gamma = \Z^d$ and is used in the proof of Theorems~\ref{t:th4} and~\ref{t:poly-abel} above.
Section~\ref{sec:proofs} is devoted to  the proof of all the theorems stated in the Introduction.
The final section discusses some  examples and open questions.

We thank Hanfeng Li and Klaus Schmidt as well as the anonymous referee for valuable comments and remarks.

\section{Background material and preliminary results}
\label{sec:background}

\subsection{Notation}
We write
$\C$ for the complex numbers, $\R$ for the reals,  $\T \coloneqq \{z \in \C : |z| = 1\}$ for the unit circle group and $\N$ for the set of non-negative integers.
\par
Depending on the context, multiplicative or additive notation is used for groups.
We prefer  additive notation for abelian groups except for the circle group $\T$ for which we shall adopt  multiplicative notation.
\par
All group actions considered in this paper are on the left.
If $\Gamma$ is a group acting on sets $X$ and $Y$, one says that a map $f \colon X \to Y$ is
$\Gamma$-\emph{equivariant} if $f(\gamma x) = \gamma f(x)$ for all $\gamma \in \Gamma$ and 
$x \in X$. 

\subsection{Affine maps}
Let $X,Y$ be topological groups and let $f \colon X \to Y$ be a map.
One says that $f$ is an \emph{affine map} if there is a continuous group morphism $a \colon X \to Y$ and an element $b \in Y$ such that
\begin{equation}
\label{e:affine}
f(x) = a(x) \cdot b \quad \text{for all } x \in X. 
\end{equation}
This amounts to saying that there exists a continuous group morphism 
$c \colon X \to Y$ and an element $d \in Y$ such that
$f(x) = d \cdot c(x)$ for all $x \in X$ 
(write $a(x) \cdot b = b \cdot (b^{-1} a(x) b)$ and $d \cdot c(x) = (d c(x) d^{-1})\cdot  d$).
\par

\begin{proposition}
\label{p:adcc-aff-surj}
Let $X$ be a  compact metrizable  group equipped with an action of a countable group $\Gamma$ by continuous group automorphisms. 
Suppose that $(X,\Gamma)$ satisfies the algebraic descending chain condition.
Then every injective $\Gamma$-equivariant affine map $f \colon X \to X$ is surjective.
\end{proposition}

\begin{proof}
The proof is analogue to the one given in the Introduction for the surjunctivity of dynamical systems satisfying the descending chain condition.
More specifically, let $f \colon X \to X$ be an  affine map.
Let $a \colon X \to X$  be a  continuous group endomorphism and $b \in X$ 
as in~\eqref{e:affine}.
Observe that $f$ is injective (resp.~surjective) if and only if
$a$ is injective (resp.~surjective), and that $f$ is $\Gamma$-equivariant if and only if
$a$ is $\Gamma$-equivariant and $b$ is fixed by $\Gamma$.
Suppose that $f$ is $\Gamma$-equivariant.
  By applying the a.d.c.c.\ to the sequence
  of closed $\Gamma$-invariant subgroups
  $$
  X= a^0(X) \supset a(X) = a^1(X) \supset a^2(X) \supset \dots,
  $$
  we see that there is an integer $k \geq 0$ such that
  $a^k(X) = a^{k + 1}(X)$.
  Now if $f$ is injective, then $a$ and hence $a^k$ are also injective,
so that we get $X = a(X)$.
This implies that $a$ and hence $f$ are surjective.
\end{proof}

An algebraic dynamical system $(X,\Gamma)$ is said to be \emph{topologically rigid} if every 
$\Gamma$-equivariant continuous  map 
$\tau \colon X \to X$ is affine.
As an immediate consequence of Proposition~\ref{p:adcc-aff-surj}, we get the following.

\begin{corollary}
\label{c:rigid-adcc-are-surj}
Let $X$ be a  compact metrizable  group equipped with an action of a countable group $\Gamma$ by
continuous group automorphisms. 
Suppose that $(X,\Gamma)$ is topologically rigid and satisfies the algebraic descending chain condition.
Then $(X,\Gamma)$ is surjunctive.
\qed
\end{corollary}

\subsection{Pontryagin duality}
\label{subsec:pontryagin}
We use the monograph of Morris~\cite{morris} as a general reference for Pontryagin duality.
\par
By an  \emph{lca group}, we mean an abelian topological group that is locally compact and Hausdorff.
Let $X$ be an lca group.
A continuous group morphism $\chi \colon X \to \T$ is called a \emph{character} of $X$.
The set $\widehat{X}$ of all characters of $X$,
equipped with pointwise multiplication and the topology of uniform convergence on compact sets,
is  an lca group \cite[Theorem~10]{morris}, called the \emph{character group} or  \emph{Pontryagin dual}  of $X$.
The celebrated Pontryagin - van Kampen duality theorem (cf.\ \cite[Theorem~23]{morris})
asserts that
the evaluation map $\iota \colon X \to \widehat{\widehat{X}}$,
defined by $\iota(x)(\chi) \coloneqq \chi(x)$ for all $x \in X$ and $\chi \in \widehat{X}$,
yields an  isomorphism of topological groups between $X$ and its bidual 
$\widehat{\widehat{ X}}$. 
This  can be used to identify any lca group with its bidual. 
The space $X$ is
compact (resp.~discrete, resp.~metrizable, resp.~$\sigma$-compact)
if and only if $\widehat{X}$ is
discrete (resp.~compact, resp.~$\sigma$-compact, resp.~metrizable)
(see Theorems 12 and 19 in \cite{morris}).
In particular, $X$ is compact and metrizable if and only if $\widehat{X}$ is discrete and countable.
When $X$ is compact, $X$ is connected if and only if $\widehat{X}$ is a \emph{torsion-free} group  (i.e., has no non-trivial elements of finite order) \cite[Corollary 4 in Chapter 7]{morris}.
When $X$ is compact and connected, 
the topological dimension of $X$ is equal to the rank of
$\widehat{X}$
(cf.\ \cite[Theorem 34]{morris}).
Here the \emph{rank} of an abelian group is defined as being the maximal cardinality of a linearly independent subset.  
\par
Let $X$ be an lca group and $Y \subset X$ a closed subgroup.
Then the set 
$$
\ann(Y) \coloneqq \{\chi \in \widehat{X}: \chi(y) =1 \text{  for all  } y \in Y\}
$$
is a closed subgroup of $\widehat{X}$,
called the \emph{annihilator} of $Y$.
After identifying $X$ with its bidual, one has
$\ann(\ann(Y)) = Y$
\cite[Proposition~38]{morris}.
Moreover, as topological groups,
$\widehat{Y}$ is canonically isomorphic to  $\widehat{X}/\ann(Y)$ and
$\ann(Y)$ is canonically isomorphic to  $\widehat{X/Y}$
(cf.\ \cite[Theorem~27]{morris}).
\par
Let $X, Y$ be lca groups and $f \colon X \to Y$ a continuous group morphism.
The map $\widehat{f} \colon \widehat{Y} \to \widehat{X}$, defined by 
$\widehat{f}(\chi) \coloneqq  \chi \circ f$ for all $\chi \in \widehat{Y}$ is a continuous group morphism, called the \emph{dual} of $f$.
After identifying $X$ and $Y$ with their biduals,
we have that $\widehat{\widehat{f}} =f $.
If $f$ is surjective, then $\widehat{f}$ is injective.
On the other hand, if $f$ is both injective and open, then $\widehat{f}$ is surjective.
As a consequence, if $X$ and $Y$ are either both compact or both discrete, then  $f$ is injective (resp.~surjective) if and only if $\widehat{f}$ is surjective (resp.~injective)  \cite[Proposition~30]{morris}.

\begin{proposition}
\label{p:char-group-mor}
Let $X,Y$ be lca groups and $f \colon X \to Y$ a continuous map.
Then $f$ is a group morphism if and only if 
$\xi \circ f \in \widehat{X}$ for all $\xi \in \widehat{Y}$.
\end{proposition}

\begin{proof}
The necessity is obvious since the composite of two continuous group morphisms is also a continuous group morphism.
Conversely, suppose that $\xi \circ f \in \widehat{X}$ for all $\xi \in \widehat{Y}$.
Then, we have that
\begin{equation*}
\xi(f(x_1 + x_2)) = \xi \circ f (x_1 + x_2) = \xi \circ f (x_1) \cdot  \xi \circ f(x_2) = \xi(f(x_1) + f(x_2))
\end{equation*}
for all $x_1,x_2 \in X$ and $\xi \in \widehat{Y}$.
As it follows  from the Pontryagin - van Kampen duality theorem that characters separate points in $Y$, we deduce that $f(x_1 + x_2) = f(x_1) + f(x_2)$.
This shows that $f$ is a group morphism.
\end{proof}

\begin{corollary}
\label{c:char-affine}
Let $X,Y$ be lca groups and $f \colon X \to Y$ a continuous map.
Then $f$ is affine if and only if
$\xi \circ f$ is affine for every $\xi \in \widehat{Y}$.
\end{corollary}

\begin{proof}
This immediately follows from Proposition~\ref{p:char-group-mor} after observing that $f$ is affine if and only if $f - f(0_X)$ is a group morphism.
\end{proof}

\subsection{Fourier transform}
Let $X$ be a compact metrizable abelian group with Haar measure $\mu$ normalized by 
$\mu(X) = 1$.
Recall that its Pontryagin dual $\widehat{X}$ is a discrete countable abelian group.
Let $L^2(X) \coloneqq L^2(X,\mu)$ denote the Hilbert space of (equivalence classes of) complex-valued functions on $X$ that are square-integrable with respect to $\mu$.
The inner product on $L^2(X)$ is given by
$$
\langle h_1 , h_2 \rangle = \int_X h_1(x) \overline{h_2(x)} \, d\mu(x) 
$$
for all $h_1, h_2 \in L^2(X)$.
Note that there is a natural inclusion  $\widehat{X} \subset L^2(X)$ since $\T \subset \C$. 
Actually $\widehat{X}$ is an Hilbert basis of $L^2(X)$
(this is a particular case of the Peter-Weyl theorem, see e.g. \cite[Sections~21-22]{weil} and 
\cite[Chapter~XI, Section~26]{godement}).
\par
Let  $\ell^2(\widehat{X})$ denote the Hilbert space of square-summable complex-valued functions on 
$\widehat{X}$.
Thus $\ell^2(\widehat{X}) = L^2(\widehat{X},\nu)$, where $\nu$ is the counting measure on $\widehat{X}$.
A Hilbert basis for $\ell^2(\widehat{X})$ is given by the family $(e_\chi)_{\chi \in \widehat{X}}$, where 
$e_\chi \in \ell^2(\widehat{X})$ is the map that takes the value $1$ at $\chi$ and $0$ everywhere else. 
\par
Since $\widehat{X}$ is a Hilbert basis of $L^2(X)$,
there is a unique unitary isomorphism
$\FF_X \colon L^2(X)  \to \ell^2(\widehat{X})$, called the \emph{Fourier transform},
that satisfies 
 $\FF_X(\chi) = e_\chi$ for all $\chi \in \widehat{X}$.
Given $h \in L^2(X)$, one has
\[
(\FF_X(h))(\chi) \coloneqq \langle h, \chi \rangle 
\]
for all $\chi \in \widehat{X}$ and $\FF_X(h)$ is called the \emph{Fourier transform} of $h$.
\par
 Suppose now that $X$ is equipped with an action of a countable group $\Gamma$ by continuous group automorphisms.
 Then there is a   unitary action of $\Gamma$ on $L^2(X)$ given by
 $$
 (\gamma h)(x) \coloneqq h(\gamma^{-1} x)
 $$
 for all $\gamma \in \Gamma$,  $h \in L^2(X)$, and $x \in X$. Note that $\widehat{X}$ is a $\Gamma$-invariant subset of $L^2(X)$.
 The action of $\Gamma$ on $\widehat{X}$ also induces a unitary action of $\Gamma$ on 
 $\ell^2(\widehat{X})$ given by
 $$
 (\gamma k)(\chi) \coloneqq  k(\gamma^{-1} \chi)
 $$
 for all $\gamma \in \Gamma$,  $k \in \ell^2(\widehat{X})$, and $\chi \in \widehat{X}$.
Observe that
 $\gamma e_\chi = e_{\gamma \chi}$ for all $\gamma \in \Gamma$ and $\chi \in \widehat{X}$.
 This shows in particular that the Fourier transform $\FF_X \colon L^2(X)  \to \ell^2(\widehat{X})$
 is $\Gamma$-equivariant.
  Finally note that these various actions of $\Gamma$ induce a $\Z[\Gamma]$-module structure on $X$ and 
  $\widehat{X}$ and a $\C[\Gamma]$-module structure on $L^2(X)$ and $\ell^2(\widehat{X})$.
The Fourier transform $\FF_X$ is a $\C[\Gamma]$-module isomorphism from $L^2(X)$ 
onto $\ell^2(\widehat{X})$.

\subsection{The van Kampen lifting theorem}
Given topological abelian groups $X$ and $K$,
we denote by $C(X,K)$ the set of all continuous maps $h \colon X \to K$ and by $C_0(X,K)$ the subset of $C(X,K)$ consisting of all $h \in C(X,K)$ such that $h(0_X) = 0_K$.
We shall need the following result, due to van Kampen~\cite{van-kampen} (see also~\cite{walters-compact-groups}, \cite{bhattacharya-ward}).    

\begin{theorem}[van Kampen's lifting theorem]
\label{t:van-kampen-lift}
Let $X$ be a compact connected metrizable abelian group and $q\in C(X, \T)$.
Then there exist unique elements $t_q \in \T$, $\chi_q \in {\widehat X}$, and
$h_q\in C_0(X,\R)$ such that
\begin{equation}
\label{e:vanKampen}
q(x) = t_q \chi_q(x) e^{2\pi i h_q(x)},
\end{equation}
for all $x\in X$. 
\qed
\end{theorem}

Note that $C(X,\T)$ (resp.~$C_0(X,\R)$) is an abelian group for pointwise multiplication (resp.~pointwise addition).
When $X$ is endowed with an action of a countable group $\Gamma$ by continuous group automorphisms,  $\Gamma$ acts by group automorphisms on $C(X,\T)$ (resp.~$C_0(X,\R)$) by 
$(\gamma f)(x) \coloneqq f(\gamma^{-1} x)$  for all $x \in X$ and
$f \in C(X,\T)$ (resp.~$f \in C_0(X,\R)$).
This gives a $\Z[\Gamma]$-module structure on $C(X,\T)$ and $C_0(X,\R)$.
We shall use the following fact.

\begin{lemma}
\label{l:h-equivariante}
Let $X$ be a compact connected metrizable abelian group
equipped with an action of a countable group $\Gamma$ by continuous group automorphisms.
Then the map from $C(X,\T)$ into $C_0(X,\R)$
given by $q\mapsto h_q$, where $h_q$ is defined by~\eqref{e:vanKampen}, 
is a $\Z[\Gamma]$-module morphism.
\end{lemma}

\begin{proof}
Let $q,q'\in C(X, \T)$ and $\gamma \in \Gamma$. 
 Then~\eqref{e:vanKampen} gives us, for all $x \in X$,
 \begin{equation*}
\label{e:vanKampen2}
q(x) = t_q \chi_q(x) e^{2\pi i h_q(x)}
\quad \text{and} \quad
q'(x) = t_{q'} \chi_{q'}(x) e^{2\pi i h_{q'}(x)}.
\end{equation*}
Therefore
\begin{align*}
(q q')(x) &= q(x) q'(x) = 
t_q \chi_q(x) e^{2\pi i h_q(x)}
t_{q'} \chi_{q'}(x) e^{2\pi i h_{q'}(x)} \\
&= t_q t_{q'} (\chi_{q} \chi_{q'})(x) e^{2\pi i (h_q + h_{q'})(x)}
\end{align*}
and
\[
(\gamma q)(x) = q(\gamma^{-1} x) 
= t_q \chi_q(\gamma^{-1} x) e^{2\pi i h_q(\gamma^{-1} x)} 
= t_q (\gamma \chi_q)(x) e^{2\pi i (\gamma h_q)(x)}.
\]
By applying the uniqueness part of the van Kampen theorem, we deduce that
$h_{q q'} = h_q + h_{q'}$ and $h_{\gamma q} = \gamma h_q$.
This shows that the map $h \mapsto h_q$ is a $\Z[\Gamma]$-module morphism.
\end{proof}

\subsection{The $\ell^2$-zero-divisor conjecture}
Let $\Gamma$ be a countable group.
Denote by $\C[\Gamma]$ the group algebra of $\Gamma$ over $\C$.
We recall that $\C[\Gamma]$ is the vector space of finitely-supported maps  
$f \colon \Gamma \to \C$ with the convolution product defined by
\begin{equation}
\label{e:convol}
(f g)(\gamma) \coloneqq 
\sum_{\substack{\gamma_1, \gamma_2 \in \Gamma:\\ \gamma_1\gamma_2 = \gamma}}  f(\gamma_1) g(\gamma_2)
\end{equation}
for all $f,g \in \C[\Gamma]$ and $\gamma \in \Gamma$.
\par
Let $\ell^2(\Gamma)$ denote the Hilbert space of square-summable maps  
$k \colon \Gamma \to \C$.
We equip $\ell^2(\Gamma)$ with the \emph{regular representation} of $\Gamma$, i.e., the unitary action of $\Gamma$ given by
$$
(\gamma k)(\gamma') = k(\gamma^{-1} \gamma')
$$
for all $k \in \ell^2(\Gamma)$ and $\gamma,\gamma' \in \Gamma$.
This defines a $\C[\Gamma]$-module structure on $\ell^2(\Gamma)$ which extends the $\C[\Gamma]$-module structure on the ring $\C[\Gamma] \subset \ell^2(\Gamma)$ viewed as a left-module over itself.
\par
One says that $\Gamma$ satisfies Kaplansky's   \emph{zero-divisor conjecture} (over~$\C$) if the ring 
$\C[\Gamma]$ has no zero-divisors (cf. \cite[Chapter~13]{passman}).
This  amounts to saying that $\C[\Gamma]$ is torsion-free as a $\C[\Gamma]$-module.
One says that $\Gamma$ satisfies the
\emph{$\ell^2$-zero-divisor conjecture} if $\ell^2(\Gamma)$ is torsion-free as a $\C[\Gamma]$-module \cite[Conjecture~8.1]{linnell-l2-zdc}.
Clearly $\Gamma$ satisfies the zero-divisor conjecture if it satisfies the $\ell^2$-zero-divisor conjecture.
Note that if $\gamma \in \Gamma$ has finite order $n \geq 2$, then $1 - \gamma$ is a 
two-sided zero-divisor in $\C[\Gamma]$ since
$$
(1 - \gamma)(1 + \gamma + \dots + \gamma^{n - 1}) = (1 + \gamma + \dots + \gamma^{n - 1})(1 - \gamma) = 0.
$$
Consequently, every group that satisfies the zero-divisor conjecture must be torsion-free.   
\par

\subsection{Mixing actions}
Let $X$ be a compact metrizable abelian group with normalized Haar measure $\mu$.
\par
If $\Gamma$ is a countable group acting by continuous group automorphisms on $X$, it is known that $(X,\Gamma)$ is mixing if and only if
\begin{equation}
\label{e:mix-f-g}
\lim_{\gamma \to \infty}
\langle \gamma f, g \rangle = 
\langle f, 1 \rangle \cdot \langle 1, g \rangle,
\end{equation}
for all $f,g \in L^2(X)$. Here  $1 \in \widehat{X} \subset L^2(X)$ is the trivial character of $X$.
The following result is an immediate consequence of
Theorem~1.6 in~\cite{schmidt-book}
(see also~\cite[Proposition~4.1]{lind-schmidt-survey-heisenberg}).
We include a proof for the convenience of the reader.

\begin{lemma}
\label{l:inj-orbit-mixing}
Let  $X$ be a compact metrizable abelian group equipped with an action of a countable group 
$\Gamma$ by continuous group automorphisms.
Suppose that $(X,\Gamma)$ is mixing and that the group $\Gamma$ is torsion-free.
Let $\chi$ be a non-trivial character of $X$.
Then the orbital map from $\Gamma$ into $\widehat{X}$ given by $\gamma \mapsto \gamma \chi$  is injective.
In other words, the stabilizer of $\chi$ in $\Gamma$ is trivial.
\end{lemma}

\begin{proof}
As $\widehat{X}$ is a Hilbert basis of $L^2(X)$,
we have that
$$
\langle \chi, \chi \rangle = 1 \quad \text{and} \quad \langle \chi , 1 \rangle = 0.
$$
Let $\gamma \in \Gamma$ with $\gamma \not= 1_\Gamma$. 
Since $\Gamma$ is torsion-free, $\gamma$ has infinite order and hence
$ \gamma^n$ tends to infinity as  $|n| \to \infty$.
Thus, taking $f = \chi$ and $g = 1$ in~\eqref{e:mix-f-g}, we get
\[
\lim_{|n| \to \infty}
\langle \gamma^n \chi, \chi \rangle = \langle \chi, 1 \rangle \cdot \langle 1, \chi \rangle = 0.
\]
This implies $\gamma \chi \not= \chi$ since otherwise we would have
$\langle \gamma^n \chi , \chi \rangle = \langle \chi, \chi \rangle = 1$ for all $n$. 
\end{proof}

\begin{proposition}
\label{p:k-*}
Let  $X$ be a compact metrizable abelian group equipped with an action of a countable group 
$\Gamma$ by continuous group automorphisms.
Suppose that the group $\Gamma$ is torsion-free and that the dynamical system $(X,\Gamma)$ is mixing.
Let $\chi$ be a non-trivial character of $X$. 
For $k\in \ell^2({\widehat X})$, define a map $k_\chi \colon \Gamma \to \C$ by setting
\begin{equation}
\label{e:k-sub-*}
k_\chi(\gamma) \coloneqq  k(\gamma \chi)
\end{equation}
for all $\gamma \in \Gamma$.
Then the following hold:
\begin{enumerate}[{\rm (i)}]
\item $k_\chi \in \ell^{2}(\Gamma)$ for all $k \in \ell^2(\widehat{X})$;
\item the map from $\ell^2({\widehat X})$ to $\ell^2(\Gamma)$ given by $k\mapsto k_\chi$  is a 
$\C[\Gamma]$-module morphism.
\end{enumerate}
\end{proposition}

\begin{proof}
Let $k \in \ell^2({\widehat X})$.
It follows from Lemma~\ref{l:inj-orbit-mixing} that the map $\gamma \mapsto \gamma \chi$ is injective. As a consequence, we have that
\[
\sum_{\gamma \in \Gamma} \vert k_\chi(\gamma) \vert^2 
= \sum_{\gamma \in \Gamma} \vert k(\gamma \chi) \vert^2 
\leq
\sum_{\chi' \in \widehat{X}} \vert k(\chi') \vert^2  
< \infty.
\] 
This shows (i). 
\par
On the other hand, for all $\gamma, \gamma' \in \Gamma$, we have that
\[
(\gamma k)_\chi(\gamma') 
= (\gamma k)(\gamma' \chi) = k(\gamma^{-1}\gamma' \chi) =
k_\chi(\gamma^{-1} \gamma') = (\gamma k_\chi)(\gamma').
\]
This shows that $(\gamma k)_\chi = \gamma k_\chi$.
Therefore the map $k\mapsto k_\chi$ is $\Gamma$-equivariant 
and (ii) follows by linearity.
\end{proof}

\section{Topological rigidity and the $\ell^2$-zero-divisor conjecture}
\label{sec:l2-zdc}

The goal of this section is to establish the following  rigidity result.

\begin{theorem}
\label{t:rigidity-l2}
Let $X$ and $Y$ be compact connected metrizable abelian groups
equipped with an action of a countable group $\Gamma$ by continuous group automorphisms.
Suppose that the following conditions hold:
\begin{enumerate}[{\rm (i)}]
\item $\Gamma$ satisfies the $\ell^2$-zero-divisor conjecture;
\item $(X,\Gamma)$ is mixing;
\item ${\widehat Y}$ is a torsion $\Z[\Gamma]$-module.
\end{enumerate}
Then every $\Gamma$-equivariant continuous map $f \colon X\rightarrow Y$ is affine.
\end{theorem}

\begin{remark}
In the particular case $\Gamma = \Z^d$, Theorem~\ref{t:rigidity-l2}
can be deduced  from \cite[Theorem~1.1]{bhattacharya-ward}.
\end{remark}

\begin{remark}
When $Y$ is a compact metrizable abelian group and $\Gamma$ is  a countable amenable group acting on $Y$  by continuous group automorphisms,
it is known (cf. \cite[Theorem~2.2]{deninger-entropy}) that the topological entropy of $(Y, \Gamma)$ coincides with its measure-theoretic entropy with respect to the Haar measure on $Y$.
Furthermore, when $\Gamma = \Z^d$ and $\widehat{Y}$ is finitely generated as a $\Z[\Gamma]$-module, 
it is shown in \cite[Lemma~2.2]{bhattacharya-ward} that $(Y,\Gamma)$ has finite topological entropy 
if and only if $\widehat{Y}$ is a torsion $\Z[\Gamma]$-module.
When $\Gamma$ is a general countable amenable group,
it is still true that $\widehat{Y}$ is a torsion $\Z[\Gamma]$-module if $(Y,\Gamma)$ has finite topological entropy.
Indeed, if $\chi \in \widehat{Y}$ has no $\Z[\Gamma]$-torsion, then the cyclic $\Z[\Gamma]$-submodule $M \subset \widehat{Y}$ generated by $\chi$ is isomorphic to $\Z[\Gamma]$. 
Thus $(\widehat{M},\Gamma)$ is conjugate to the $\Gamma$-shift on $\T^\Gamma$ and hence has infinite topological entropy.
As $(\widehat{M},\Gamma)$ is a factor of $(Y,\Gamma)$, this implies that $(Y,\Gamma)$ has infinite topological entropy as well.   
\end{remark}

For the proof of Theorem~\ref{t:rigidity-l2}, we shall use the following auxiliary result.

\begin{lemma}
\label{l:c0-x-c-tf}
Let $X$ be a compact metrizable abelian group equipped with an action of a countable group $\Gamma$ by continuous group automorphisms.
We equip $C_0(X,\C)$ with the $\C[\Gamma]$-module structure induced by the inclusion 
$C_0(X,\C) \subset L^2(X)$.
Suppose that the group $\Gamma$ satisfies the $\ell^2$-zero-divisor conjecture and that 
$(X,\Gamma)$ is mixing. 
Then  $C_0(X,\C)$ is torsion-free as a $\C[\Gamma]$-module.
\end{lemma} 

\begin{proof}
Let  $h \in C_0(X,\C)$ such that $p h = 0$ for some non-zero element $p \in \C[\Gamma]$. 
We want to show that $h = 0$.
As remarked above, $h \in L^{2}(X)$.
As the Fourier transform $\FF_X \colon L^2(X) \to \ell^2(\widehat{X})$ is a $\C[\Gamma]$-module morphism,
we have that
$$
p\FF_X(h) = \FF_X(p  h) = \FF_X(0)  = 0.
$$
Now fix a non-trivial character $\chi$ of $X$ and consider the element
$(\FF_X(h))_\chi \in \ell^2(\Gamma)$ defined as in~\eqref{e:k-sub-*}.
Using the fact that  the map from $\ell^2(\widehat{X})$ to $\ell^2(\Gamma)$ given by 
$k \mapsto k_\chi$ is a $\C[\Gamma]$-module morphism (see Proposition~\ref{p:k-*}.(ii)),
we get
\[
p( \FF_X(h))_\chi) = (p  \FF_X(h))_\chi = 0_\chi = 0.
\] 
Since $\Gamma$ satisfies the $\ell^2$-zero-divisor conjecture and $p \not= 0$, this 
implies that $(\FF_X (h))_\chi = 0$.
We deduce that $(\FF_X(h))(\chi) = ((\FF_X(h))_\chi)(1_\Gamma) = 0$.
As $\chi$ was an arbitrary non-trivial character of $X$, it follows that $h$ is a constant.
But $h(0_X) = 0$ since $h \in C_0(X,\C)$. Therefore  $h = 0$.
\end{proof}

\begin{proof}[Proof of Theorem~\ref{t:rigidity-l2}]
Let $f \colon X\to Y$ be a $\Gamma$-equivariant continuous  map.
We want to prove that $f$ is affine.
By Corollary~\ref{c:char-affine},
it is enough to show that $\xi \circ f$ is affine for every character $\xi$ of $Y$.
So let us fix some character $\xi \in \widehat{Y}$ and consider the map
$\xi \circ f \in C(X,\T)$.
By the van Kampen lifting theorem (cf. Theorem~\ref{t:van-kampen-lift}), there exist
 $t \in \T$, $\chi \in {\widehat X}$, and
$h\in C_0(X,\R)$ such that
\begin{equation}
\label{e:vanKampen-2}
\xi \circ f (x) = t \chi(x) e^{2\pi i h(x)}
\end{equation}
for all $x \in X$.
\par
On the other hand, since $\widehat{Y}$ is a torsion $\Z[\Gamma]$-module,
there is a  
non-zero element $p\in \Z[\Gamma]$
such that $p \xi = 0$.
\par
As $f$ is $\Gamma$-equivariant, we have, for all $\gamma \in \Gamma$ and $x \in X$,
\[
(\gamma(\xi \circ f))(x) = \xi(f(\gamma^{-1}x)) = \xi(\gamma^{-1}f(x)) = ((\gamma \xi) \circ f)(x).
\]
This shows that $\gamma(\xi \circ f) = (\gamma \xi) \circ f$. 
By linearity, one has more generally $r(\xi \circ f) = (r \xi) \circ f$ for all $r \in \Z[\Gamma]$.
\par
We deduce that $p (\xi \circ f) = 0$.
\par
By applying Lemma~\ref{l:h-equivariante}, 
we deduce that $p h = 0$.
It follows that  $h = 0$ by Lemma~\ref{l:c0-x-c-tf}.
Formula~\eqref{e:vanKampen-2} now becomes
$$
\xi \circ f (x) = t \chi(x) \quad \text{for all  }x \in X.
$$
This shows that the map $\xi \circ f$ is affine and completes the proof. 
\end{proof}

As an immediate consequence of Theorem~\ref{t:rigidity-l2}, we get the following.

\begin{corollary}
\label{c:rigidity-l2}
Let $X$  be a compact connected metrizable abelian group
equipped with an action of a countable group $\Gamma$ by continuous group automorphisms.
Suppose that the following conditions hold:
\begin{enumerate}[{\rm (i)}]
\item $\Gamma$ satisfies the $\ell^2$-zero-divisor conjecture;
\item $(X,\Gamma)$ is mixing;
\item ${\widehat X}$ is a torsion $\Z[\Gamma]$-module.
\end{enumerate}
Then the algebraic dynamical system $(X,\Gamma)$ is topologically rigid.
\qed
\end{corollary}

\section{Proofs}
\label{sec:proofs}
In this section, we provide the proofs of the statements presented in the Introduction.  
For the proof of Theorem~\ref{t:th1}, we shall use the following result.

\begin{lemma}
\label{l:inj-endo-surj}
Let $X$ be a compact connected metrizable abelian group with finite topological dimension.
Then every injective continuous group endomorphism of $X$ is surjective.
\end{lemma}

\begin{proof}
Let $a \colon X \to X$ be an injective continuous group morphism and let $n$ denote the topological dimension of $X$.
Then the Pontryagin dual $\widehat{X}$ of $X$ is a discrete torsion-free group with rank $n$.
Therefore $V \coloneqq \widehat{X} \otimes_\Z \Q$ is an $n$-dimensional $\Q$-vector space and
$\widehat{X}$ embeds as a subgroup of $V$ via the map $\chi \mapsto  \chi \otimes_\Z 1$.
Furthermore, the dual endomorphism $\widehat{a}$ uniquely extends to a $\Q$-linear map 
$$\widehat{a}_\Q \coloneqq a \otimes_\Z \Id_{\Q} \colon V \to V.$$
The injectivity of $a$ implies the
surjectivity of $\widehat{a}$
(see~\cite[Proposition~30]{morris}) 
and hence the surjectivity of $\widehat{a}_\Q$.
Since every surjective endomorphism of a finite-dimensional vector space is injective,  it follows that
$\widehat{a}_\Q$ is injective.
As $\widehat{a}$ is the restriction of $\widehat{a}_\Q$ to $\widehat{X}$, we deduce that 
$\widehat{a}$ is itself injective and therefore we conclude that
$a = \widehat{\widehat{a}}$ is surjective.
  \end{proof}

\begin{proof}[Proof of Theorem~\ref{t:th1}]
As pointed out in the Introduction, the group $X$ is abelian by \cite[Theorem~3.2]{lam}.
Let $\tau \colon X \to X$ be an injective   $\Gamma$-equivariant continuous map.
It follows from  a rigidity result of Bhattacharya \cite[Corollary~1]{bhattacharya}  that $\tau$ is an affine map. Therefore there exist a continuous group endomorphism
$a  \colon X \to X$ and an element $b \in X$ such that
$\tau(x) = a(x) + b$ for all $x \in X$.
The injectivity of $\tau$ implies    that of $a$.
Applying Lemma~\ref{l:inj-endo-surj}, we deduce that $a$ is surjective.
This  implies that $\tau$ is surjective as well. 
\end{proof}

\begin{proof}[Proof of Theorem~\ref{t:th2}]
It follows from a result of Kitchens and Schmidt
\cite[Corollary~7.4]{kitchens-schmidt-1989} that, under theses hypotheses, the $\Z^d$-periodic points are dense in $X$.
On the other hand, as already mentioned in the Introduction,  it is known 
\cite[Proposition~5.1]{csc-ijm-2015} that every expansive dynamical system 
admitting a dense set of periodic points is surjunctive. 
\end{proof}

\begin{proof}[Proof of Theorem~\ref{t:th3}]
Let us set $\Gamma \coloneqq \Z^d$ and let $\tau \colon X \to X$ be an injective   $\Gamma$-equivariant continuous map.
Let $\Lambda \subset \Gamma$ be a subgroup of finite index and denote by
$X(\Lambda)\subset X$ the closed subgroup consisting  
of all the points of $X$ that are fixed by  $\Lambda$. 
Observe that $X(\Lambda)$ is  $\Gamma$-invariant since $\Gamma$ is abelian.
As a $\Z[\Gamma]$-module, 
the Pontryagin dual
$\widehat{X(\Lambda)}$
is a quotient of $\widehat{X}$.
On the other hand, since $(X,\Gamma)$ satisfies the a.d.c.c.\  by our hypotheses, the $\Z[\Gamma]$-module $\widehat{X}$ is Noetherian and hence finitely generated.
Therefore $\widehat{X(\Lambda)}$ 
 is a finitely generated $\Z[\Gamma]$-module. As $\Lambda$ is of finite index in $\Gamma$, we deduce that  $\widehat{X(\Lambda)}$ is
also finitely generated as a $\Z[\Lambda]$-module.
This is the same as saying that $\widehat{X(\Lambda)}$ is a finitely generated abelian group
since the action of $\Lambda$ on $X(\Lambda)$ is trivial.
It follows that there exist an integer $k \geq 0$ and a finite abelian group $F$ such that the group 
$\widehat{X(\Lambda)}$
is isomorphic to $\Z^k \times F$.
By dualizing,
we deduce that $X(\Lambda)$ is homeomorphic to $\T^k \times F$.
Now, $\tau$ is $\Gamma$-equivariant so that $\tau(X(\Lambda))  \subset X(\Lambda)$. 
As $\tau$ is injective and $\T^k \times F$ is incompressible, it follows that $\tau(X(\Lambda)) = X(\Lambda)$.
Since the union of the sets $X(\Lambda)$, as $\Lambda$ varies over all finite index subgroups of $\Gamma = \Z^d$, is dense by \cite[Theorem 7.2]{kitchens-schmidt-1989} (see also \cite[Theorem  11.2]{schmidt-book}), we conclude  that the map $\tau$ is surjective.
\end{proof}

\begin{proof}[Proof of Theorem~\ref{t:th4}]
The algebraic dynamical system $(X,\Gamma)$ is topologically rigid
by Corollary~\ref{c:rigidity-l2} and satisfies the a.d.c.c.\ since $\widehat{X}$ is assumed to be Noetherian as a $\Z[\Gamma]$-module.
Therefore $(X,\Gamma)$ is surjunctive by Corollary~\ref{c:rigid-adcc-are-surj}.   
\end{proof}

\begin{proof}[Proof of Theorem~\ref{t:poly-abel}]
Suppose first  that $\Gamma$ is polycyclic-by-finite.
Then $\Gamma$ is solvable-by-finite and hence elementary amenable.
On the other hand,  the integral group ring $\Z[\Gamma]$ is Noetherian (see \cite{hall}, \cite{passman}). 
As $\widehat{X}$ is finitely generated as a $\Z[\Gamma]$-module, 
it follows that $\widehat{X}$ is Noetherian.
Thus $(X,\Gamma)$ is surjunctive by Theorem~\ref{t:th4}.
\par
Suppose now that $\Gamma$ is abelian.
Let $f \colon X \to X$ be an injective $\Gamma$-equivariant continuous map. 
Since every  abelian group is elementary amenable,
we deduce from  Theorem~\ref{t:rigidity-l2} 
that $f$ is affine.
Thus there exist a continuous group morphism $a \colon X \to X$ and an element $b \in X$ such that 
$f(x) = a(x) + b$ for all $x \in X$.
Observe that $a$ is injective since $f$ is.
Consequently,   its dual $\widehat{a}$  is a surjective endomorphism of the
$\Z[\Gamma]$-module $\widehat{X}$.
Now,
by a theorem of Vasconcelos \cite[Proposition~1.2]{vasconcelos} (see also 
\cite[Theorem~2.4]{matsumura} for a proof based on Nakayama's lemma), it is known that if $R$ is a commutative ring, then
every surjective endomorphism of a finitely generated $R$-module  is injective.
Since the  ring $\Z[\Gamma]$ is commutative, it follows that $\widehat{a}$ is injective and hence that 
$a = \widehat{\widehat{a}}$ is surjective.
This implies that $f$ itself is surjective.
\end{proof}

\section{Examples and concluding remarks}

  \begin{example}
Consider the infinite-dimensional torus    $X \coloneqq \T^\N$.
The map $a \colon X \to X$, given by
$a(x)(0) \coloneqq 0$ and $  a(x)(n) \coloneqq x(n - 1)$  for all $x \in X$ and $n \geq 1$,
   is a continuous injective group endomorphism of $X$ that is not surjective
   (this shows that the hypothesis that $X$ is finite-dimensional cannot be removed from Lemma~\ref{l:inj-endo-surj}).
Then, for any countable group $\Gamma$,  the $\Gamma$-shift $(X^\Gamma,\Gamma)$ yields an example of a non-surjunctive mixing algebraic dynamical system whose phase space 
is connected and abelian.
  \end{example}
  
 \begin{example}
Let $p$ be any prime and  thake $X \coloneqq \Z_p$, the group of $p$-adic integers.
Then $X$ is a $0$-dimensional compact metrizable abelian group.    
The map $a \colon X \to X$, defined by  $a(x) \coloneqq p x$ for all $x \in X$, is an injective continuous group endomorphism of $X$ that is not surjective
   (this shows that the hypothesis that $X$ is connected cannot be removed from 
   Lemma~\ref{l:inj-endo-surj}).
For any countable group $\Gamma$,  the $\Gamma$-shift $(X^\Gamma,\Gamma)$ yields an example of a non-surjunctive mixing algebraic dynamical system whose phase space 
is $0$-dimensional and abelian.
\end{example}

\begin{remark}
Suppose that $\Gamma$ is a countable residually finite group and  $S$ a compact metrizable space. Then  periodic points for the $\Gamma$-shift are dense in $S^\Gamma$.
Indeed, let $x \in S^\Gamma$ and $N \subset S^\Gamma$ a neighborhood of $x$.
By definition of the product topology, there is a finite set $\Omega \subset \Gamma$ such that $N$ contains all $y \in S^\Gamma$ that satisfy $y|_\Omega = x|_\Omega$.
Since $\Gamma$ is residually finite, we can find a finite index subgroup $\Lambda$ of $\Gamma$ such that
no two distinct elements of $\Omega$ belong to the same right coset of $\Lambda$.
Consequently, we can find a complete set of representatives $R \subset \Gamma$ for the right cosets of $\Lambda$ such that $\Omega \subset R$.
Then the point  $y \in S^\Gamma$,
defined by 
   $y(\lambda r) \coloneqq x(r)$ for all $\lambda \in \Lambda$ and $r \in R$,
   is in $N$. On the other hand, $y$ is periodic since it is fixed by $\Lambda$.
This proves that periodic points are dense in $S^\Gamma$.
Using the two previous Examples,
 we deduce that for any countable residually finite group $\Gamma$ (e.g. $\Gamma = \Z^d$), there exist
mixing  algebraic dynamical systems $(X,\Gamma)$ 
with a dense set of periodic points and  $X$ connected (resp.~$0$-dimensional)
that are not surjunctive.
\end{remark}

\begin{example}[Principal algebraic actions]
Let $\Gamma$ be a countable group and let $f \in \Z[\Gamma]$.
Consider the cyclic left $\Z[\Gamma]$-module  $M_f \coloneqq  \Z[\Gamma]/\Z[\Gamma] f$
obtained by quotienting the integral group ring $\Z[\Gamma]$ by the principal left ideal generated by $f$ and let $X_f$ denote the Pontryagin dual of $M_f$.
Recall that $X_f$ is the set of all characters $x \colon M_f \to \T$ with the topology of pointwise convergence and that there is a continuous action of $\Gamma$ on 
$X_f$  satisfying  $(\gamma x) (m) = x(\gamma^{-1} m)$ for all $\gamma \in \Gamma$, $x \in X_f$, and $m \in M_f$.
Observe that $X_f$ is the closed shift-invariant subset of $\T^\Gamma$ given by
$$
X_f = \{ x \in \T^\Gamma : \sum_{\gamma \in \Gamma} f(\gamma' \gamma) x(\gamma) = 0 \text{ for all } \gamma' \in \Gamma \}.
$$
(here we  identify  the unit circle $\T$ with $\R/\Z$).
One says that $(X_f,\Gamma)$ is the \emph{principal} algebraic dynamical system associated with $f$
(see e.g.~\cite{einsiedler-rindler}, \cite{deninger-schmidt}, \cite{lind-schmidt-survey-heisenberg}).
There are several cases where we can say that
the algebraic dynamical system $(X_f,\Gamma)$ is surjunctive.
\par
Suppose first that $\Gamma = \Z^d$.
Then $\Z[\Gamma]$ can be identified with the ring $R_d \coloneqq \Z[u_1, u_1^{-1}, \dots, u_d, u_d^{-1}]$ of Laurent polynomials with integral coefficients on $d$ commuting indeterminates.
Since the ring $R_d$ is Noetherian, the algebraic dynamical system $(X_f,\Gamma)$ satisfies the a.d.c.c.\ so that it is surjunctive for every $f \in R_d$ by Theorem~\ref{t:th3}.
For $f = 1 + u_1 + u_2 \in R_2$, we get the \emph{connected Ledrappier subshift}, which is mixing but not expansive (cf. \cite[Example~5.6]{lind-schmidt-handbook}, \cite{schmidt-book}).
\par
Take again an arbitrary countable group $\Gamma$.
Then the vector space  $\ell^1(\Gamma)$ of all complex-valued summable functions on 
$\Gamma$ is a Banach algebra under the convolution product (cf. Formula~\eqref{e:convol})
containing $\Z[\Gamma]$ as a subring. 
Suppose from now on that $f \in \Z[\Gamma] $ is invertible in $\ell^1(\Gamma)$
(this is for example the case when  $f$ is \emph{lopsided}, i.e., there exists an element  
$\gamma_0 \in \Gamma$ such that
$|f(\gamma_0)| > \sum_{\gamma \not= \gamma_0} |f(\gamma)|$, 
cf.~\cite{lind-schmidt-survey-heisenberg}).
Then it is known that $(X_f,\Gamma)$ is expansive \cite[Theorem~3.2]{deninger-schmidt}.
   If in addition $\Gamma$ is residually finite (e.g. finitely generated nilpotent or, more generally, 
   polycyclic-by-finite), then the periodic points are dense in $X_f$
   by the first assertion in Proposition~8.4 in~\cite{lind-schmidt-survey-heisenberg} 
   (note that the hypothesis that $\Gamma$ is amenable is not needed for the proof therein), so that
    $(X_f,\Gamma)$ is surjunctive by~\cite[Proposition~5.1]{csc-ijm-2015}.
    On the other hand,  the fact that $f$ is invertible in $\ell^1(\Gamma)$
    also implies that $(X_f,\Gamma)$ is mixing
(cf.~\cite[Proposition~4.6]{lind-schmidt-survey-heisenberg}).
As $M_f$ is a torsion $\Z[\Gamma]$-module as soon as $f \not= 0$, we  deduce from Theorem~\ref{t:poly-abel} that $(X_f,\Gamma)$ is surjunctive whenever $X_f$ is connected and $\Gamma$ is a torsion-free abelian group  (e.g.~$\Gamma = \Q$). 
        \end{example}

It would be interesting to find examples of algebraic dynamical systems that are expansive or that satisfy the algebraic descending chain condition but are not surjunctive.

\end{document}